\documentclass[12pt,a4paper]{article}
\usepackage[ngerman, english]{babel}
\usepackage{enumerate}
\usepackage{amsmath,amssymb, amsthm, amsfonts}
\usepackage[colorlinks]{hyperref}

\theoremstyle{plain}
\newtheorem {lemma}{Lemma}
\newtheorem {proposition}[lemma]{Proposition}
\newtheorem {theorem}[lemma]{Theorem}

\theoremstyle{definition}
\newtheorem{definition}[lemma]{Definition}
\newtheorem {example}[lemma]{Example}
\newcommand{\N}{\mathbb{N}}
\newcommand{\Z}{\mathbb{Z}}

\newcommand{\tr}{\operatorname{tr}}
\newcommand{\GL}{\operatorname{GL}}
\newcommand{\SL}{\operatorname{SL}}
\newcommand{\E}{\operatorname{E}}
\newcommand{\C}{\operatorname{C}}
\newcommand{\Mat}{\operatorname{M}}
\newcommand{\F}{\mathbb{F}}

\title{On products of conjugacy classes in general linear groups}
\author{Raimund Preusser}
\date{}
\AtEndDocument{\bigskip{\footnotesize%
  \textsc{Chebyshev Laboratory, St. Petersburg State University, Russia} \par  
  \textit{E-mail address:} \texttt{raimund.preusser@gmx.de} \par
}}

\begin{document}
\maketitle
\begin{abstract}
\noindent
Let $K$ be a field and $n\geq 3$. Let $\E_n(K)\leq H\leq \GL_n(K)$ be an intermediate group and $C$ a noncentral $H$-class. Define $m(C)$ as the minimal positive integer $m$ such that $\exists i_1,\dots,i_m\in\{\pm 1\}$ such that the product $C^{i_1}\dots C^{i_m}$ contains all nontrivial elementary transvections. In this article we obtain a sharp upper bound for $m(C)$. Moreover, we determine $m(C)$ for any noncentral $H$-class $C$ under the assumption that $K$ is algebraically closed or $n=3$ or $n=\infty$.
\end{abstract}
\let\thefootnote\relax
\let\thefootnote\relax
\let\thefootnote\relax\footnotetext{The work is supported by the Russian Science Foundation grant 19-71-30002.}

\section{Introduction}
In 1964, H.\ Bass \cite{bass} showed that if $R$ is an (associative, unital) ring then 
\begin{equation}
\forall H\leq\GL_n(R):H^{\E_n(R)}\subseteq H\Leftrightarrow \exists I\lhd R:\E_n(R,I)\leq H\leq \C_n(R,I)
\end{equation}
provided $n$ is large enough with respect to the stable rank of $R$. Here $\E_n(R)$ denotes the elementary subgroup, $\E_n(R,I)$ the relative elementary subgroup of level $I$ and $\C_n(R,I)$ the full congruence subgroup of level $I$. Bass's result is known as {\it Sandwich Classification Theorem}. In the 1970's and 80's, the validity of this theorem was extended by J.\ Wilson \cite{wilson}, I.\ Golubchik \cite{golubchik}, L.\ Vaserstein \cite{vaserstein, vaserstein_ban, vaserstein_neum} and others. Statement (1) holds true for example if $R$ is a commutative ring and $n\geq 3$.

It follows from the Sandwich Classification Theorem that if $\sigma$ is a matrix in $\GL_n(R)$ where $R$ is a commutative ring and $n\geq 3$, then each of the elementary transvections $t_{kl}(\sigma_{ij}), t_{kl}(\sigma_{ii}-\sigma_{jj})~(i\neq j,k\neq l)$ can be expressed as a finite product of $\E_n(R)$-conjugates of $\sigma$ and $\sigma^{-1}$. 
In 1960, J.\ Brenner \cite{brenner} showed that in the case $R=\Z$ there is a bound for the number of $\E_n(R)$-conjugates needed for such an expression. In 2018, the author \cite{preusser2} proved that indeed for \textit{any} commutative ring there is a bound for the number of $\E_n(R)$-conjugates needed for such an expression. In 2020, the author \cite{preusser11} obtained bounds for different classes of noncommutative rings.

Now consider the case that $R=K$ is a field and $n\geq 3$. It follows from the Sandwich Classification Theorem that for any noncentral $\sigma\in \GL_n(K)$ the elementary transvection $t_{12}(1)$ can be expressed as a finite product of $\E_n(K)$-conjugates of $\sigma$ and $\sigma^{-1}$. Rephrased, for any noncentral $\E_n(K)$-class $C$ (see beginning of Section 2) there is an $m\in\N$ and $i_1,\dots,i_m\in\{\pm 1\}$ such that $T\subseteq C^{i_1}\dots C^{i_m}$ where $T$ is the $\E_n(K)$-class of $t_{12}(1)$. What is the smallest $m$ with this property? This paper aims to answer this question. More generally, if $\E_n(K)\leq H\leq \GL_n(K)$ is an intermediate group and $C$ a noncentral $H$-class, we define
\[m(C):=\min\{m\in\N\mid \exists i_1,\dots,i_m\in\{\pm 1\}: T\subseteq C^{i_1}\dots C^{i_m}\}
\]
where $T$ denotes the $H$-class of $t_{12}(1)$. The Sandwich Classification Theorem implies that the minimum in the definition of $m(C)$ exist. The goal of this paper is to compute $m(C)$.

The rest of the paper is organised as follows. In Section 2 we recall some standard notation which is used throughout the paper. Moreover, we recall the Frobenius form and the generalised Jordan form of a square matrix and some basic definitions regarding general linear groups. In Section 3 we prove that $m(C)\leq 4$ for any noncentral $H$-class $C$ and that this bound is sharp (see Proposition \ref{propfour} and Example \ref{exeisen}). Moreover, we show that if $K$ is algebraically closed, then $m(C)\leq 2$. In Section 4 we compute $m(C)$ in the case $n=3$, the main result of this section is Theorem \ref{thmn=3}. In Section 5 we consider the case $n=\infty$. 

\section{Preliminaries}
If $G$ is a group and $g,h\in G$, we let $g^h:=h^{-1}gh$ and $[g,h]:=ghg^{-1}h^{-1}$. Let $H$ be a subgroup of $G$ and $g,g'\in G$. If there is an $h\in H$ such that $g^h=g'$, we write $g\sim_H g'$ (or just $g\sim g'$ if $H=G$). Clearly $\sim_H$ is an equivalence relation on $G$. We call the $\sim_H$-equivalence class of an element $g\in G$ the \textit{$H$-class} of $g$ and denote it by $g^{H}$. For an $H$-class $C=g^H$ we define $C^1:=C$ and $C^{-1}:=(g^{-1})^H$.

Throughout the paper $K$ denotes a field and $K^*$ the set of all nonzero elements of $K$. $\N$ denotes the set of positive integers.
For any $m,n\in\N$, the set of all $m\times n$ matrices over $K$ is denoted by $\Mat_{m\times n}(K)$. Instead of $\Mat_{n\times n}(K)$ we may write $\Mat_n(K)$. The identity matrix in $\Mat_n(K)$ is denoted by $e$ or $e_{n\times n}$ and the matrix with a one at position $(i,j)$ and zeros elsewhere is denoted by $e^{ij}$. 

\subsection{Normal forms for matrices}
In this subsection $n$ denotes a positive integer. We recall the Frobenius form and the generalised Jordan form of a matrix in $\Mat_n(K)$. 

\begin{definition}
Let $P=X^n+a_{n-1}X^{n-1}+\dots+a_1X+a_0\in K[X]$ be a monic polynomial of degree $n$. The \textit{companion matrix} of $P$ is the matrix
\[[P]=\begin{pmatrix}&&&-a_0\\1&&&-a_1\\&\ddots&&\vdots\\&&1&-a_{n-1}\end{pmatrix}\in \Mat_{n}(K).\]
\end{definition} 

\begin{definition}
Let $\sigma\in \Mat_m(K)$ and $\tau\in \Mat_n(K)$. The \textit{direct sum} of $\sigma$ and $\tau$ is the matrix
\[\sigma\oplus\tau=\begin{pmatrix}\sigma&\\&\tau\end{pmatrix}\in \Mat_{m+n}(K).\]
\end{definition}

\begin{definition}
Let $\sigma,\tau\in \Mat_n(K)$. If there is an invertible $\rho\in\Mat_n(K)$ such that $\sigma=\rho\tau\rho^{-1}$, then we write $\sigma\sim\tau$ and call $\sigma$ and $\tau$ \textit{similar}.
\end{definition}

\begin{theorem}\label{thmfro}
Let $\sigma\in\Mat_n(K)$. Then there are uniquely determined nonconstant, monic polynomials $P_1,\dots,P_r\in K[X]$ such that $P_1|P_2|\dots|P_r$ and $\sigma \sim [P_1]\oplus\dots\oplus [P_r]$.
\end{theorem}
\begin{proof}
See \cite[Part V, Chapter 21, Theorem 4.4]{bjn}.
\end{proof}

The matrix $[P_1]\oplus\dots\oplus [P_r]$ in Theorem \ref{thmfro} is called the \textit{Frobenius form} or \textit{rational canonical form} of $\sigma$. We denote it by $F(\sigma)$. The polynomials $P_1,\dots,P_r$ are called the \textit{invariant factors} of $\sigma$.

Recall that the \textit{characteristic matrix} of a matrix $\sigma\in\Mat_n(K)$ is the matrix $Xe-\sigma\in \Mat_n(K[X])$. The \textit{characteristic polynomial} of $\sigma$ is the polynomial $\chi_\sigma=\det(Xe-\sigma)$. Note that the characteristic polynomial of a companion matrix $[P]$ is $P$. The invariant factors of a matrix $\sigma\in\Mat_n(K)$ are precisely the monic polynomials associated with the nonconstant polynomials in the Smith normal form of the characteristic matrix $Xe-\sigma$. The Smith normal form of $Xe-\sigma$ can be computed using the algorithm described in \cite[Part V, Chapter 20, Proof of Theorem 3.2]{bjn}.

Let $\sigma\in\Mat_n(K)$ and $P_1,\dots,P_r$ the invariant factors of $\sigma$. For each $1\leq i\leq r$ we can write $P_i=\prod\limits_{j=1}^{s_i}P_{ij}^{q_{ij}}$ where  $P_{i1},\dots,P_{is_i}\in K[X]$ are pairwise distinct irreducible, monic polynomials and $q_{i1},\dots,q_{is_i}\geq 1$. This factorisation is unique up to the order of the factors. The polynomials $P_{ij}^{q_{ij}}~(1\leq i\leq r, 1\leq j \leq s_i)$ form the system of \textit{elementary divisors} of $\sigma$. Note that $P_{ij}^{q_{ij}}$ maybe equal to $P_{i'j'}^{q_{i'j'}}$ if $i\neq i'$.

In order to define the generalised Jordan form of a matrix, we need the notion of the generalised Jordan block of a power $P^q$ of an irreducible polynomial $P$. For polynomials $P$ of degree $1$ one gets back the usual Jordan blocks.

\begin{definition}
Let $q\in\N$ and $P\in K[X]$ an irreducible, monic polynomial of degree $n$. The \textit{generalised Jordan block} corresponding to $P^q$ is the matrix 
\[J(P^q)=\begin{pmatrix}[P]&&&\\\xi&[P]&&\\&\ddots&\ddots&\\&&\xi&[P]\end{pmatrix}\in\Mat_{qn}(K)\]
where $\xi\in \Mat_n(K)$ is the matrix that has a $1$ at position $(1,n)$ and zeros elsewhere.
\end{definition}

\begin{theorem}\label{thmjor}
Let $\sigma\in\Mat_n(K)$ and $(P_{i}^{q_{i}})_{i\in\Phi}$ the system of elementary divisors of $\sigma$. Then $\sigma\sim \bigoplus_{i\in\Phi}J(P_i^{q_{i}})$ where the order of direct summands maybe arbitrary.
\end{theorem}
\begin{proof}
See \cite[Part V, Chapter 21, Section 5]{bjn}.
\end{proof}

The matrix $\bigoplus_{i\in\Phi}J(P_{i}^{q_{i}})$ in Theorem \ref{thmjor} is called the \textit{generalised Jordan form} or \textit{primary rational canonical form} of $\sigma$. It is uniquely determined up to the order of the generalised Jordan blocks. If the characteristic polynomial of $\sigma$ splits into linear factors, then one gets back the usual Jordan form of $\sigma$.

\subsection{The general linear group $\GL_n(K)$} 

In this subsection $n$ denotes a positive integer. 
\begin{definition}
The group $\GL_{n}(K)$ consisting of all invertible elements of $\Mat_n(K)$ is called the {\it general linear group} of degree $n$ over $K$.
\end{definition}

\begin{definition}
Let $a\in K$ and $1\leq i\neq j\leq n$. Then the matrix $t_{ij}(a)=e+ae^{ij}$ is called an {\it elementary transvection}. If $a\neq 0$, then $t_{ij}(a)$ is called \textit{nontrivial}. The subgroup $\E_n(K)$ of $\GL_n(K)$ generated by the elementary transvections is called the {\it elementary subgroup}.
\end{definition}

\begin{lemma}\label{propdecomp}
For any $\sigma\in\GL_n(K)$ there is an $\epsilon\in\E_n(K)$ such that $\sigma=\epsilon d_n(\det(\sigma))$. It follows that $\E_n(K)$ equals the subgroup $\SL_n(K)$ of $\GL_n(K)$ consisting of all matrices with determinant $1$.
\end{lemma}
\begin{proof}
See for example \cite[\S 5]{bass}.
\end{proof}

The lemma below is easy to check.
\begin{lemma}\label{lemelrel}
The relations
\begin{align*}
t_{ij}(a)t_{ij}(b)&=t_{ij}(a+b), \tag{R1}\\
[t_{ij}(a),t_{hk}(b)]&=e \tag{R2}\text{ and}\\
[t_{ij}(a),t_{jk}(b)]&=t_{ik}(ab) \tag{R3}\\
\end{align*}
hold where $i\neq k, j\neq h$ in $(R2)$ and $i\neq k$ in $(R3)$.
\end{lemma}

\begin{definition}\label{defp}
Let $a\in K^*$ and $1\leq i\neq j\leq n$. We define
\[d_{i}(a):=e+(a-1)e^{ii}\in \GL_{n}(K)\]
and 
\begin{align*}
d_{ij}(a):=e+(a-1)e^{ii}+(a^{-1}-1)e^{jj}\in \SL_n(K)=\E_{n}(K).
\end{align*}
Moreover, we define 
\[p_{ij}:=e+e^{ij}+e^{ji}-e^{ii}-e^{jj}\in \GL_{n}(K)\]
and 
\[\hat p_{ij}:=e+e^{ij}-e^{ji}-e^{ii}-e^{jj}\in\SL_n(K)=\E_{n}(K).\]
\end{definition}

\subsection{The stable general linear group $\GL_\infty(K)$} 
\begin{definition}
The direct limit $G_\infty(K)=\varinjlim_{n}\GL_n(K)$ with respect to the transition homomorphisms $\phi_{n,n+k}:\GL_n(K)\rightarrow \GL_{n+k}(K),~\sigma\mapsto e_{k\times k}\oplus \sigma  $ is called the \textit{stable general linear group} over $K$.
\end{definition}

Let $\sigma\in \GL_m(K)$ and $\tau\in \GL_n(K)$. If $\phi_{m,\max\{m,n\}}(\sigma)=\phi_{n,\max\{m,n\}}(\tau)$, we write $\sigma\sim_\infty\tau$. We identify $\GL_\infty(K)$ with the set $\bigcup_{n}\GL_n(K)/\sim_\infty$ of all $\sim_\infty$-equivalence classes made a group by defining $[\sigma]_\infty[\tau]_\infty=[\phi_{m,\max\{m,n\}}(\sigma)\phi_{n,\max\{m,n\}}(\tau)]_\infty$ for any $\sigma\in \GL_m(K)$ and $\tau\in \GL_n(K)$. 

\begin{definition}
The subgroup $\E_\infty(K)$ of $\GL_\infty(K)$ generated by the elements $[t_{n,i,j}(a)]_\infty$ $(n\geq 2,~1\leq i\neq j\leq n,~a\in K)$ is called the \textit{elementary subgroup}. Here $t_{n,i,j}(a)$ denotes the elementary transvection $t_{ij}(a)\in \GL_n(K)$.
\end{definition}

\section{Products of conjugacy classes in $\GL_n(K)$}
In this section $n$ denotes an integer greater than $2$. We set $G:=\GL_n(K)$ and $E:=\E_n(K)$. $H$ denotes a subgroup of $G$ containing $E$, and $T$ denotes the $H$-class of $t_{12}(1)$. If $C$ is a noncentral $H$-class, then $m(C)$ is defined as in Section 1.

The lemma below shows that $T=t_{12}(1)^G=t_{12}(1)^E$. 
\begin{lemma}\label{lem2}
Let $\sigma\in G$. Then $t_{12}(1)\sim_E \sigma\Leftrightarrow t_{12}(1)\sim_H \sigma\Leftrightarrow t_{12}(1)\sim_G \sigma$.
\end{lemma}
\begin{proof}
Clearly $t_{12}(1)\sim_E \sigma\Rightarrow t_{12}(1)\sim_H \sigma\Rightarrow t_{12}(1)\sim_G \sigma$. Hence it suffices to show that $t_{12}(1)\sim_G \sigma\Rightarrow t_{12}(1)\sim_E \sigma$. Suppose that $t_{12}(1)\sim_G \sigma$. Then there is a $\rho\in G$ such that $t_{12}(1)=\sigma^\rho$. By Lemma \ref{propdecomp} there is an $\epsilon\in E$ such that $\rho=\epsilon d_n(\det(\rho))$. Since $n\geq 3$, the matrix $t_{12}(1)$ commutes with $d_n(\det(\rho))$ and hence $t_{12}(1)=\sigma^\epsilon$. Thus $t_{12}(1)\sim_E\sigma$.
\end{proof}

It follows from the next lemma that $T$ contains all nontrivial elementary transvection.

\begin{lemma}\label{lem1}
Let $t_{ij}(a)$ and $t_{kl}(b)$ be nontrivial elementary transvections. Then $t_{ij}(a)$ $\sim_E t_{kl}(b)$.
\end{lemma}
\begin{proof}
It is an easy exercise to show that $t_{ij}(a)^\epsilon=t_{kl}(a)$ for some 
$\epsilon\in\E$ which is a product of $\hat p_{st}$'s (see Definition \ref{defp}). Hence $t_{ij}(a)\sim_E t_{kl}(a)$. It remains to show that $t_{kl}(a)\sim_E t_{kl}(b)$. But clearly $t_{kl}(a)^{d_{lm}(a^{-1}b)}=t_{kl}(b)$ for any $m\neq k,l$. 
\end{proof}

Let $C$ be an $H$-class. Since similar matrices have the same characteristic polynomial (resp. determinant, trace, invariant factors, elementary divisors), we can define $\chi_C$ (resp. $\det(C)$, $\tr(C)$, the invariant factors of $C$, the elementary divisors of $C$, the Frobenius form $F(C)$) in the obvious way. Below we compute $F(T)$. Note that $T=\{g\in G\mid F(\sigma)=F(T)\}$ by Lemma \ref{lem2}.

\begin{lemma}\label{lemfroel}
$F(T)=[X-1]\oplus\dots\oplus [X-1]\oplus[(X-1)^2]$.
\end{lemma}
\begin{proof}
By Lemma \ref{lem1}, we have $F(T)=F(t_{n,n-1}(1))$. One checks easily that 
\[t_{n,n-1}(1)^{t_{n-1,n}(1)}= \left(\begin{array}{c|cc}e_{(n-2)\times(n-2)}&&\\\hline &&-1\\&1&2\end{array}\right).\]
\end{proof}

Let $C$ and $D$ be noncentral $H$-classes. We write $C\sim D$ and call $C$ and $D$ \textit{conjugated} if there is a $\rho\in G$ such that $C^\rho=D$. The lemma below implies that $m(C)$ does only depend on the conjugacy class of $C$.

\begin{lemma}\label{lem3}
Let $C\sim D$ be conjugated noncentral $H$-classes and suppose that $T\subseteq C^{i_1}\dots C^{i_k}$ where $i_1,\dots,i_k\in\{\pm 1\}$. Then $T\subseteq D^{i_1}\dots D^{i_k}$.
\end{lemma}
\begin{proof}
Choose a $\sigma \in C$ and a $\tau\in D$. Since $C\sim D$, there is a $\rho\in G$ such that $\sigma=\tau^\rho$. By Lemma \ref{propdecomp} there is an $\epsilon\in E$ and an $a\in K^*$ such that $\rho=\epsilon d_n(a)$. Hence
\begin{equation}
\sigma=\tau^{\epsilon d_n(a)}.
\end{equation}
Since $T\subseteq C^{i_1}\dots C^{i_k}$, there are $\rho_1,\dots,\rho_k\in H$ such that 
\begin{equation*}
t_{12}(1)=(\sigma^{i_1})^{\rho_1}\dots (\sigma^{i_k})^{\rho_k}.
\end{equation*}
By Lemma \ref{propdecomp} there are $\epsilon_1,\dots,\epsilon_k\in E$ and $a_1,\dots,a_k\in K$ such that $\rho_i=\epsilon_i d_n(a_i)~(1\leq i \leq k)$. Note that $d_n(a_1),\dots,d_n(a_k)\in H$ since $H$ contains $E$. We have 
\begin{equation}
t_{12}(1)=(\sigma^{i_1})^{\epsilon_1 d_n(a_1)}\dots (\sigma^{i_k})^{\epsilon_k d_n(a_k)}.
\end{equation}
It follows from Equations (2) and (3) that 
\begin{equation*}
t_{12}(1)=(\tau^{i_1})^{\epsilon d_n(a)\epsilon_1 d_n(a_1)}\dots (\tau^{i_k})^{\epsilon d_n(a)\epsilon_k d_n(a_k)}.
\end{equation*}
One easily checks that $d_n(a)$ commutes with elementary transvections modulo $E$. Hence there are $\epsilon'_1,\dots,\epsilon'_k\in E$ such that $d_n(a)\epsilon_i=\epsilon_i'd_n(a)~(1\leq i\leq k)$. We get 
\begin{equation*}
t_{12}(1)=(\tau^{i_1})^{\epsilon \epsilon'_1 d_n(a_1)d_n(a)}\dots (\tau^{i_k})^{\epsilon \epsilon'_k d_n(a_k)d_n(a)}.
\end{equation*}
Since $n\geq 3$, the matrix $t_{12}(1)$ commutes with $d_n(a)$ and thus
\begin{equation*}
t_{12}(1)=(\tau^{i_1})^{\epsilon \epsilon'_1 d_n(a_1)}\dots (\tau^{i_k})^{\epsilon \epsilon'_k d_n(a_k)}\in D^{i_1}\dots D^{i_k}.
\end{equation*}
\end{proof}

\begin{lemma}\label{lemm=1}
Let $C$ be a noncentral $H$-class. Then $m(C)=1$ iff $C=T$.
\end{lemma}
\begin{proof}
Clear since $T=T^{-1}$ by Lemma \ref{lem1}.
\end{proof}

Let $C$ be a noncentral $H$-class. Proposition \ref{proproot} below shows that if $\chi_C$ has a root in $K$ (which is always true if $K$ is algebraically closed), then $m(C)\leq 2$. But we will see later that $m(C)$ can be greater than $2$ if $\chi_C$ has no root. 

\begin{proposition}\label{proproot}
Let $C$ be a noncentral $H$-class. If $\chi_C$ has a root in $K$, then $T\subseteq CC^{-1}$.
\end{proposition}
\begin{proof}
Let $a\in K$ be a root of $\chi_C$. First note that $a\neq 0$ since the constant coefficient of $\chi_C$ is $(-1)^n\det(C)\neq 0$. Since $(X-a)$ divides $\chi_C$ we get that $(X-a)^q$ is an elementary divisor of $C$ for some $q\in\N$ (since $\chi_C$ is the product of the elementary divisors of $C$). Choose a $\sigma\in C$. By Theorem \ref{thmjor} we have 
\[\sigma\sim \begin{pmatrix} J((X-a)^q)&\\&\tau\end{pmatrix}=:\rho \]
for some $\tau\in \GL_{n-q}(K)$. Recall that 
\[J((X-a)^q)=\begin{pmatrix}a&&&\\1&a&&\\&\ddots&\ddots&\\&&1&a\end{pmatrix}\in\GL_{q}(K).\]
By Theorem \ref{thmfro} we may assume that $\tau$ is in Frobenius form, i.e. there are nonconstant, monic polynomials $P_1,\dots,P_r\in K[X]$ such that $P_1|P_2|\dots|P_r$ and $\tau=[P_1]\oplus\dots\oplus [P_r]$.
\begin{enumerate}[{\bf Case} 1]
\item Suppose that $q=1$. Then 
\[\rho=\begin{pmatrix} a&\\&\tau\end{pmatrix}.\]
\begin{enumerate}[{\bf Subcase}{ 1.}1] 
\item Suppose that each of the $P_i$'s has degree $1$. Since $P_1|P_2|\dots|P_r$, it follows that $P_1=\dots=P_r=X-b$ for some $b\in K^*$. Hence 
\[\rho=\begin{pmatrix} a&&&\\&b&&\\&&\ddots&\\&&&b\end{pmatrix}.\]
One checks easily that $[\rho,t_{12}(1)]=t_{12}(ab^{-1}-1)$. It follows that $t_{12}(ab^{-1}-1)\in DD^{-1}$ where $D=\rho^H$. Since $\sigma$ is noncentral, we have $a\neq b$ and hence $ab^{-1}-1\neq 0$. Hence $T\subseteq DD^{-1}$ by Lemma \ref{lem1} and thus $T\subseteq CC^{-1}$ by Lemma \ref{lem3}.
\item Suppose there is an $i$ such that $P_i$ has degree $t\geq 2$. Write $P_i=X^t+b_{t-1}X^{t-1}+\dots+b_1X+b_0$. Then 
\[\sigma\sim \left(\begin{array}{c|ccccc|c}a&&&&&&\\\hline &&&&&-b_0&\\&1&&&&-b_1&\\&&1&&&-b_2&\\&&&\ddots&&\vdots&\\&&&&1&-b_{t-1}&\\\hline&&&&&&*\end{array}\right)=:\xi.\]
It follows that 
\[\sigma\sim \xi^{t_{23}(-a)}=\left(\begin{array}{c|ccccc|c}a&&&&&&\\\hline &a&*&&&*&\\&1&*&&&*&\\&&1&&&-b_2&\\&&&\ddots&&\vdots&\\&&&&1&-b_{t-1}&\\\hline&&&&&&*\end{array}\right)=:\zeta.\]
One checks easily that $[\zeta,t_{21}(1)]=t_{31}(a^{-1})$. It follows as in Subcase 1.1 that $T\subseteq CC^{-1}$.
\end{enumerate}
\item Suppose that $q=2$. Then 
\[\rho= \left(\begin{array}{cc|c} a&&\\1&a&\\\hline&&\tau\end{array}\right).\]
\begin{enumerate}[{\bf Subcase}{ 2.}1] 
\item Suppose that each of the $P_i$'s has degree $1$. Since $P_1|P_2|\dots|P_r$, it follows that $P_1=\dots=P_r=X-b$ for some $b\in K^*$. Hence 
\[\rho= \left(\begin{array}{cc|ccc} a&&&&\\1&a&&&\\\hline&&b&&\\&&&\ddots&\\&&&&b\end{array}\right).\]
Clearly
\[\sigma\sim\rho^{t_{12}(a-b)}= \left(\begin{array}{cc|ccc} b&*&&&\\1&*&&&\\\hline&&b&&\\&&&\ddots&\\&&&&b\end{array}\right)=:\xi.\]
One checks easily that $[\xi,t_{13}(1)]=t_{23}(b^{-1})$. It follows as in Subcase 1.1 that $T\subseteq CC^{-1}$.
\item Suppose there is an $i$ such that $P_i$ has degree $t\geq 2$. Write $P_i=X^t+b_{t-1}X^{t-1}+\dots+b_1X+b_0$. Then 
\[\sigma\sim \left(\begin{array}{cc|ccccc|c}a&&&&&&&\\1&a&&&&&&\\\hline &&&&&&-b_0&\\&&1&&&&-b_1&\\&&&1&&&-b_2&\\&&&&\ddots&&\vdots&\\&&&&&1&-b_{t-1}&\\\hline&&&&&&&*\end{array}\right)=:\xi.\]
It follows that 
\[\sigma\sim \xi^{t_{34}(-a)}=\left(\begin{array}{cc|ccccc|c}a&&&&&&&\\1&a&&&&&&\\\hline &&a&*&&&*&\\&&1&*&&&*&\\&&&1&&&-b_2&\\&&&&\ddots&&\vdots&\\&&&&&1&-b_{t-1}&\\\hline&&&&&&&*\end{array}\right)=:\zeta.\]
One checks easily that $[\zeta,t_{31}(1)]=t_{41}(a^{-1})$. It follows as in Subcase 1.1 that $T\subseteq CC^{-1}$.
\end{enumerate}
\item Suppose that $q\geq 3$. Then 
\[\rho= \left(\begin{array}{ccc|c} a&&&\\1&a&&\\&1&a&\\\hline&&&*\end{array}\right).\]
One checks easily that $[\rho,t_{21}(1)]=t_{31}(a^{-1})$. It follows as in Subcase 1.1 that $T\subseteq CC^{-1}$.
\end{enumerate}
\end{proof}

Lemma \ref{lemm=1} and Proposition \ref{proproot} directly imply the following theorem.
\begin{theorem}
Suppose that $K$ is algebraically closed. Let $C$ be a noncentral $H$-class. Then $m(C)=1$ if $C=T$ respectively $m(C)=2$ if $C\neq T$.
\end{theorem}

The converse of Proposition \ref{proproot} does not hold as the following example shows. Suppose that $K=\mathbb{F}_2$ and $n=4$. Let $C$ be the $G$-class of $[X^4+X^2+1]$. Set
\[\sigma:=\begin{pmatrix}&1&1&\\&1&&1\\1&&1&\\&1&&\end{pmatrix}\in G\text{ and }\tau:=\begin{pmatrix}&&1&1\\&1&&1\\1&&1&\\&1&&\end{pmatrix}\in G.\]
Clearly $\sigma=t_{12}(1)\tau$ and hence $\sigma\tau^{-1}=t_{12}(1)$. We leave it to the reader to check that $F(\sigma)=F(\tau)=[X^4+X^2+1]$ which implies that $\sigma,\tau\in C$. It follows that $T\subseteq CC^{-1}$ although $\chi_C=X^4+X^2+1$ has no root in $K$.
However, a ``weak converse" of Proposition \ref{proproot} does hold. Namely if $C$ is a noncentral $H$-class such that $T\subseteq CC^{-1}$, then $\chi_C$ is reducible as Proposition \ref{propweakcon} below shows. 

In Proposition \ref{propweakcon} we use the following notation. If $\sigma\in G$, $1\leq i_1<\dots<i_k\leq n$ and $1\leq j_1<\dots<j_k\leq n$, then $\sigma^{i_1,\dots,i_k}_{j_1,\dots,j_k}$ denotes the $k\times k$ matrix whose entry at position $(s,t)$ is $\sigma_{i_s,j_t}$. Recall that $\det(\sigma^{i_1,\dots,i_k}_{i_1,\dots,i_k})$, where $1\leq i_1<\dots<i_k\leq n$, is called a \textit{principal minor of size $k$}. It is well-known that 
\[\chi_\sigma=\sum\limits_{k=0}^n (-1)^ka_kX^{n-k}
\]
where $a_k$ is the sum of all principal minors of size $k$.

\begin{proposition}\label{propweakcon}
Let $C$ be a noncentral $H$-class. If $T\subseteq CC^{-1}$, then $\chi_C$ is reducible.
\end{proposition}
\begin{proof}
Suppose there is a noncentral $H$-class $C$ such that $T\subseteq CC^{-1}$ and $\chi_C$ is irreducible. We will show that this assumption leads to a contradiction. Let $\sigma,\tau\in C$ such that $t_{12}(1)=\sigma\tau^{-1}$. It follows that 
\begin{equation}
\sigma=t_{12}(1)\tau.
\end{equation}
Clearly we may assume that 
\begin{equation}
\tau_{ij}=0 \text{ for any }i\geq 2 \text{ and } j\geq i+2
\end{equation} 
(conjugate Equation (4) by appropriate elements of $E$ commuting with $t_{12}(1)$). 
Since $\sigma,\tau\in C$, we have $\tr(\sigma)=\tr(\tau)$. Hence $\tr(\sigma)\overset{(4)}{=}\tr(t_{12}(1)\tau)=\tr(\tau)+\tau_{21}=\tr(\sigma)+\tau_{21}$ whence $\tau_{21}=0$. Suppose that $\tau_{23}=0$. Then all nondiagonal entries of $\tau$ in the second row are $0$. Therefore $\chi_\tau=\chi_C$ has a linear factor, which contradicts the assumption that $\chi_C$ is irreducible. Hence $\tau_{23}\neq 0$ and therefore we may assume that $\tau_{22}=0$ (conjugate (4) by $t_{32}(-\tau_{23}^{-1}\tau_{22})$. Hence
\[
\sigma=\begin{pmatrix}\tau_{11}&\tau_{12}&\tau_{13}&\dots&\dots&\tau_{1n}\\0&0&\tau_{23}&0&\dots&0\\\tau_{31}&\tau_{32}&\tau_{33}&\ddots&\ddots&\vdots\\\vdots&\vdots&\vdots&\ddots&\ddots&0\\\vdots&\vdots&\vdots&&\ddots&\tau_{n-1,n1}\\\tau_{n1}&\tau_{n2}&\tau_{n3}&\dots&\dots&\tau_{nn}\end{pmatrix}
\]
and, because of (4), 
\[
\tau=\begin{pmatrix}\tau_{11}&\tau_{12}&\tau_{13}+\tau_{23}&\dots&\dots&\tau_{1n}\\0&0&\tau_{23}&0&\dots&0\\\tau_{31}&\tau_{32}&\tau_{33}&\ddots&\ddots&\vdots\\\vdots&\vdots&\vdots&\ddots&\ddots&0\\\vdots&\vdots&\vdots&&\ddots&\tau_{n-1,n1}\\\tau_{n1}&\tau_{n2}&\tau_{n3}&\dots&\dots&\tau_{nn}\end{pmatrix}
\]
Consider the statement 
\[P(k):~\tau_{k1}=0.\]
We will show by induction on $k$ that $P(k)$ holds for any $3\leq k\leq n$.\\
\\
\underline{$k=3$} Since $\sigma,\tau\in C$, we have $\chi_\sigma=\chi_\tau$. We compare the coefficients of $X^{n-2}$ in $\chi_\sigma$ and $\chi_\tau$, namely the sums of all principal minors of size $2$. Suppose $\det(\sigma^{i_1,i_2}_{i_1,i_2})\neq \det(\tau^{i_1,i_2}_{i_1,i_2})$. Then clearly $i_1=1$ and $i_2=3$. Moreover, $\det(\sigma^{1,3}_{1,3})-\det(\tau^{1,3}_{1,3})=-\tau_{23}\tau_{31}$. Hence $\tau_{23}\tau_{31}=0$. Since clearly $\tau_{23}\neq 0$, we obtain $\tau_{31}=0$. Thus $P(3)$ holds. \\
\\
\underline{$k\rightarrow k+1$} Suppose that $P(3),\dots,P(k)$ hold for some $k\in \{3,\dots,n-1\}$. We will show that $P(k+1)$ holds. We compare the coefficients of $X^{n-k}$ in $\chi_\sigma$ and $\chi_\tau$, namely the sums of all principal minors of size $k$ (multiplied by $(-1)^k$). Suppose $\det(\sigma^{i_1,\dots,i_{k}}_{i_1,\dots,i_{k}})\neq \det(\tau^{i_1,\dots,i_{k}}_{i_1,\dots,i_{k}})$. Then clearly $i_1=1$ and $i_2=3$. Moreover, 
\[\det(\sigma^{1,3,i_3\dots,i_{k}}_{1,3,i_3\dots,i_{k}})-\det(\tau^{1,3,i_3\dots,i_{k}}_{1,3,i_3\dots,i_{k}})=\tau_{23}\det(\sigma^{3,i_3\dots,i_{k}}_{1,i_3\dots,i_{k}})\] (consider the Laplace expansion along the first row). In view of (5) we have
\[\sigma^{3,i_3\dots,i_{k}}_{1,i_3\dots,i_{k}}=\begin{pmatrix}\tau_{31}&\tau_{3i_3}&0&\dots&0\\\tau_{i_31}&\tau_{i_3i_3}&\tau_{i_3i_4}&\ddots&\vdots\\\tau_{i_41}&\tau_{i_4i_3}&\tau_{i_4i_4}&\ddots&0\\\vdots&\vdots&\vdots&\ddots&\tau_{i_{k-1}i_k}\\\tau_{i_k1}&\tau_{i_ki_3}&\tau_{i_ki_4}&\dots&\tau_{i_ki_k}\end{pmatrix}.
\]
Consider the statement 
\[Q(j):~i_j=j+1.\]
We will show by induction on $j$ that $Q(j)$ holds for any $3\leq j\leq k$.\\
\\
\underline{$j=3$} Assume that $i_3>4$. Then $\tau_{3i_3}=0$ by (5). But $\tau_{31}=0$ by $P(3)$, whence $\det(\sigma^{1,3,i_3\dots,i_{k}}_{1,3,i_3\dots,i_{k}})-\det(\tau^{1,3,i_3\dots,i_{k}}_{1,3,i_3\dots,i_{k}})=\tau_{23}\det(\sigma^{3,i_3\dots,i_{k}}_{1,i_3\dots,i_{k}})=0$. But this contradicts the assumption that $\det(\sigma^{i_1,\dots,i_{k}}_{i_1,\dots,i_{k}})\neq \det(\tau^{i_1,\dots,i_{k}}_{i_1,\dots,i_{k}})$. Hence $i_3=4$ and thus $Q(3)$ holds.\\
\\
\underline{$j\rightarrow j+1$} Suppose that $Q(3),\dots,Q(j)$ hold for some $j\in \{3,\dots,k-1\}$. Write 
\[\sigma^{1,3,i_3\dots,i_{k}}_{1,3,i_3\dots,i_{k}}=\begin{pmatrix}A&B\\D&E\end{pmatrix}\]
where $A\in \Mat_{(j-1)\times(j-1)}(K)$, $B\in \Mat_{(j-1)\times(n-j+1)}(K)$, $D\in \Mat_{(n-j+1)\times(j-1)}(K)$ and $E\in \Mat_{(n-j+1)\times(n-j+1)}(K)$. Then 
\[A=\begin{pmatrix}\tau_{31}&\tau_{34}&0&\dots&0\\\tau_{41}&\tau_{44}&\tau_{45}&\ddots&\vdots\\\tau_{51}&\tau_{54}&\tau_{55}&\ddots&0\\\vdots&\vdots&\vdots&\ddots&\tau_{j,j+1}\\\tau_{j+1,1}&\tau_{j+1,4}&\tau_{j+1,5}&\dots&\tau_{j+1,j+1}\end{pmatrix}\]
and $B$ is the matrix whose entry at position $(j-1,1)$ is $\tau_{j+1,i_{j+1}}$ and whose other entries are zero. Assume that $i_{j+1}>j+2$. Then $\tau_{j+1,i_{j+1}}=0$ by (5). Hence $\det \sigma^{1,3,i_3\dots,i_{k}}_{1,3,i_3\dots,i_{k}}=\det(A)\det(E)$. But $\tau_{31}=\dots =\tau_{j+1,1}=0$ by $P(3),\dots, P(j+1)$ whence $\det(A)=0$. It follows that $\det(\sigma^{1,3,i_3\dots,i_{k}}_{1,3,i_3\dots,i_{k}})-\det(\tau^{1,3,i_3\dots,i_{k}}_{1,3,i_3\dots,i_{k}})=\tau_{23}\det(\sigma^{3,i_3\dots,i_{k}}_{1,i_3\dots,i_{k}})=0$. But this contradicts the assumption that $\det(\sigma^{i_1,\dots,i_{k}}_{i_1,\dots,i_{k}})\neq \det(\tau^{i_1,\dots,i_{k}}_{i_1,\dots,i_{k}})$. Hence $i_{j+1}=j+2$ and thus $Q(j+1)$ holds.\\
\\
We have shown that $Q(j)$ holds for any $3\leq j\leq k$. Hence
\[\sigma^{3,i_3\dots,i_{k}}_{1,i_3\dots,i_{k}}=\sigma^{3,4\dots,k+1}_{1,4\dots,k+1}=\begin{pmatrix}\tau_{31}&\tau_{34}&0&\dots&0\\\tau_{41}&\tau_{44}&\tau_{45}&\ddots&\vdots\\\tau_{51}&\tau_{54}&\tau_{55}&\ddots&0\\\vdots&\vdots&\vdots&\ddots&\tau_{k,k+1}\\\tau_{k+1,1}&\tau_{k+1,4}&\tau_{k+1,5}&\dots&\tau_{k+1,k+1}\end{pmatrix}.
\]
Since $\tau_{31}=\dots=\tau_{k1}=0$ by $P(3),\dots ,P(k)$, it follows that 
\begin{equation}
\tau_{23}\det(\sigma^{3,i_3\dots,i_{k}}_{1,i_3\dots,i_{k}})=(-1)^k\tau_{23}\tau_{k+1,1}\tau_{34}\tau_{45}\dots\tau_{k,k+1}.
\end{equation}
Suppose that $\tau_{j,j+1}=0$ for some $3\leq j\leq k$. Then 
\[\sigma=\begin{pmatrix}A&B&C\\0&E&0\\F&G&I\end{pmatrix}\sim \begin{pmatrix}E&0\\ *&*\end{pmatrix}
\]
where $A\in \Mat_{1\times 1}(K)$, $B\in \Mat_{1\times (j-1)}(K)$, $C\in \Mat_{1\times (n-j)}(K)$, $E\in \Mat_{(j-1)\times (j-1)}(K)$, $F\in \Mat_{(n-j)\times 1}(K)$, $G\in \Mat_{(n-j)\times (j-1)}(K)$ and $I\in \Mat_{(n-j)\times (n-j)}(K)$. It follows that $\chi_\sigma=\chi_C$ has a factor of degree $j-1$, which contradicts the assumption that $\chi_C$ is irreducible. Hence $\tau_{34},\dots,\tau_{k,k+1}\neq 0$. It follows from (6) that $\tau_{k+1,1}=0$. Thus $P(k+1)$ holds. \\
\\
We have shown that $P(k)$ holds for any $3\leq k\leq n$. Hence all nondiagonal entries of $\sigma$ in the first column are zero and hence $\chi_\sigma=\chi_C$ has a linear factor. But this contradicts the assumption that $\chi_C$ is irreducible.
\end{proof}

The next proposition implies that $m(C)\leq 4$ for any noncentral $H$-class.
\begin{proposition}\label{propfour}
Let $C$ be a noncentral $H$-class. Then $T\subseteq CC^{-1}CC^{-1}$.
\end{proposition}
\begin{proof}
Choose a $\sigma\in C$. By Theorem \ref{thmfro} and Lemma \ref{lem3} we may assume that $\sigma$ is in Frobenius form, i.e. there are nonconstant, monic polynomials $P_1,\dots,P_r\in K[X]$ such that $P_1|P_2|\dots|P_r$ and $\sigma=[P_1]\oplus\dots\oplus [P_r]$. If one of the $P_i$'s has degree $1$, then $T\subseteq CC^{-1}\subseteq CC^{-1}CC^{-1}$ by Proposition \ref{proproot} (since $\chi_C=P_1\dots P_r$). Hence we may assume that the degree of each $P_i$ is at least $2$.
\begin{enumerate}[{\bf Case} 1]
\item Suppose each $P_i$ has degree $2$. It follows that $P_1=P_2=\dots=P_r=X^2-a_1X-a_0$ for some $a_0,a_1\in K$ since $P_1|P_2|\dots|P_r$. Hence
\[\sigma=\left(\begin{array}{cc|cc|c}&a_0&&&\\1&a_1&&&\\\hline &&&a_0&\\&&1&a_1&\\\hline&&&&*\end{array}\right)\]
and
\[\sigma^{-1}=\left(\begin{array}{cc|cc|c}-a^{-1}_0a_1&1&&&\\a_0^{-1}&&&&\\\hline &&-a^{-1}_0a_1&1&\\&&a^{-1}_0&&\\\hline&&&&*\end{array}\right).\]
One checks easily that $[\sigma,t_{14}(1)]=t_{23}(a_0^{-1})t_{14}(-1)$ which implies that $[[\sigma,t_{14}(1)],$ $t_{12}(1)]=t_{13}(-a_0^{-1})$. It follows from the formula 
\begin{equation}
[[a,b],c]=a (a^{-1})^{b^{-1}}a^{b^{-1}c^{-1}} (a^{-1})^{c^{-1}},
\end{equation}
which holds for elements $a,b,c$ of any group, that $t_{13}(-a_0^{-1})\in  CC^{-1}CC^{-1}$. Thus $T\subseteq  CC^{-1}CC^{-1}$ by Lemma \ref{lem1}.
\item Suppose that there is an $i$ such that $P_i$ has degree $t\geq 3$. By Lemma \ref{lem3} we may assume that $i=1$. Write $P_1=X^t-a_{t-1}X^{t-1}-\dots-a_1X-a_0$. Then 
\[\sigma=\left(\begin{array}{ccccc|c}&&&&a_0&\\1&&&&a_1&\\&1&&&a_2&\\&&\ddots&&\vdots&\\&&&1&a_{t-1}&\\\hline&&&&&*\end{array}\right)\]
and
\[\sigma^{-1}=\left(\begin{array}{ccccc|c}-a_0^{-1}a_1&1&&&&\\-a_0^{-1}a_2&&1&&&\\\vdots&&&\ddots&&\\-a_0^{-1}a_{t-1}&&&&1&\\a_0^{-1}&&&&&\\\hline&&&&&*\end{array}\right)\]
One checks easily that $[\sigma,t_{t-1,t}(1)]=t_{t1}(a_0^{-1})t_{t-1,t}(-1)$ which implies that $[[\sigma,t_{t-1,t}(1)],t_{t-1,t}(1)]=t_{t-1,1}(-a_0^{-1})$ (since $t\geq 3$). It follows from Equation (7) that $t_{t-1,1}(-a_0^{-1})\in CC^{-1}CC^{-1}$. Thus $T\subseteq  CC^{-1}CC^{-1}$ by Lemma \ref{lem1}.
\end{enumerate}
\end{proof}

Recall that if $P=a_nX^n+\dots +a_1X+a_0\in K[X]$ is a polynomial of degree $n$, then $P^*=a_0X^n+\dots +a_{n-1}X+a_n\in K[X]$ is called the \textit{reciprocal polynomial} of $P$. If $\sigma\in G$, then $\chi_{\sigma^{-1}}=\frac{\chi_\sigma^*}{\det(\sigma)}$. It follows that $\chi_\sigma$ is irreducible iff $\chi_{\sigma^{-1}}$ is irreducible (note that $(P^*)^*=P$ and $(PQ)^*=P^*Q^*$ for any $P,Q\in K[X]$ with nonzero constant coefficient). 
\newpage

\begin{theorem}\label{thmeisen}
Let $C$ be a noncentral $H$-class. If $C\neq T$, $\chi_C$ is irreducible, $\det(C)^2\neq 1$ and $\det(C)^3\neq 1$, then $m(C)=4$.
\end{theorem}
\begin{proof}
By Lemma \ref{lemm=1} we have $m(C)>1$. By Proposition \ref{propweakcon} we have $T\not\subseteq CC^{-1}$ since $\chi_C$ is irreducible. Similarly $T\not\subseteq C^{-1}C$ since $\chi_{C^{-1}}$ is irreducible by the paragraph right before Theorem \ref{thmeisen}. Clearly $T\not\subseteq CC$ and $T\not\subseteq C^{-1}C^{-1}$ since $\det(C)^2\neq 1$. Similarly $T\not\subseteq C^{i_1}C^{i_2}C^{i_3}$ for all $i_1,i_2,i_3\in\{\pm 1\}$ since $\det(C)^3\neq 1$. Hence $m(C)>3$. It follows from Proposition \ref{propfour} that $m(C)=4$.
\end{proof}

\begin{example}\label{exeisen}
Suppose $K=\mathbb{Q}$. Let $P=X^n+2\in K[X]$. Then $P$ is irreducible by Eisenstein's criterion. Let $C$ denote the $H$-class of $[P]$. Clearly $C\neq T$ since $F(C)=[P]$ but $F(T)=[X-1]\oplus\dots\oplus [X-1]\oplus[(X-1)^2]$ by Lemma \ref{lemfroel}. Moreover, $\chi_C=P$ and $\det(C)\in\{\pm 2\}$. It follows from Theorem \ref{thmeisen} that $m(C)=4$.
\end{example}

\section{The case $n=3$}
Set $G:=GL_3(K)$ and $E:=E_3(K)$. $H$ denotes a subgroup of $G$ containing $E$, and $T$ denotes the $H$-class of $t_{12}(1)$. We will determine $m(C)$ for any noncentral $H$-class $C$.

\begin{proposition}\label{proprootn=3}
Let $C$ be a noncentral $H$-class. Then the following are equivalent.
\begin{enumerate}[(i)]
\item $\chi_C$ has a root in $K$. 
\item $T\subseteq CC^{-1}$.
\item $T\subseteq C^{-1}C$.
\end{enumerate}
\end{proposition}
\begin{proof}
The implication $(i)\Rightarrow (ii)$ follows from Proposition \ref{proproot} and the implication $(ii)\Rightarrow (i)$ from Proposition \ref{propweakcon} (note that a polynomial $P\in K[X]$ of degree $3$ is reducible iff it has a root in $K$). Applying the equivalence $(i)\Leftrightarrow (ii)$ to $C^{-1}$ we obtain 
\[\chi_{C^{-1}}\text{ has a root in }K\Leftrightarrow T\subseteq C^{-1}C.\]
It follows from the paragraph right before Theorem \ref{thmeisen} that
\[\chi_{C^{-1}}\text{ has a root in }K\Leftrightarrow \chi_C\text{ has a root in }K.\]
Thus we have shown $(i)\Leftrightarrow (iii)$.
\end{proof}

Next we will consider $H$-classes $C$ such that $\chi_C$ does not have a root in $K$.
\begin{proposition}\label{propsquaren=3}
Let $C$ be an $H$-class such that $\chi_C$ does not have a root in $K$. Then the following are equivalent.
\begin{enumerate}[(i)]
\item $\det(C)^2=1$. 
\item $T\subseteq CC$.
\item $T\subseteq C^{-1}C^{-1}$.
\end{enumerate}
\end{proposition}
\begin{proof}
We will first show  $(i)\Leftrightarrow (ii)$ and then $(i)\Leftrightarrow (iii)$.\\
The implication $(ii)\Rightarrow (i)$ is obvious since $\det (T)=1$. Suppose now that $\det(C)^2=1$. Choose a $\sigma\in C$. By Lemma \ref{lem3} we may assume that $\sigma$ is in Frobenius form. Since $\chi_C$ does not have a root in $K$, it is irreducible (because $n=3$). Hence the matrix $\sigma$ has only one invariant factor and therefore 
\[\sigma=\begin{pmatrix}&&a\\1&&b\\&1&c\end{pmatrix}\]
for some $a,b,c\in K$.
Note that $a=\det(\sigma)$ and hence $a^2=1$. Set
\[\sigma_0=\begin{pmatrix}&&a\\1&&1\\&1&-a\end{pmatrix}.\]
Then
\[\sigma_0^2=\begin{pmatrix}&a&-1\\&1&\\1&-a&2\end{pmatrix}.\]
One checks easily that $(\sigma_0^2)^{p_{12}t_{21}(a)}=[X-1]\oplus[(X-1)^2]=F(t_{12}(1))\sim t_{12}(1)$. Hence $\sigma_0^2\sim t_{12}(1)$ and therefore $\sigma_0^2\sim_E t_{12}(1)$ by Lemma \ref{lem2}. It follows that there is an $\epsilon\in E$ such that 
\begin{equation}
(\sigma_0^2)^\epsilon=t_{12}(1).
\end{equation}
Set 
\[\xi=\begin{pmatrix}-1&&-b-1\\&-1&a+c\\&&1\end{pmatrix}.\]
Clearly $\xi=t_{13}(-b-1)t_{23}(a+c)d_{12}(-1)=d_{12}(-1)t_{13}(b+1)t_{23}(-a-c)\in E$ and
\begin{equation}
\xi^2=e.
\end{equation}
One checks easily that 
\begin{equation}
\sigma_0\xi=\sigma^{d_{13}(-1)}.
\end{equation}
Clearly 
\begin{align*}
&\sigma^{d_{13}(-1)\epsilon}\sigma^{d_{13}(-1)\epsilon\xi^\epsilon}\\
\overset{(10)}{=}~&(\sigma_0\xi)^{\epsilon}(\sigma_0\xi)^{\epsilon\xi^\epsilon}\\
=~&\sigma_0^{\epsilon}\xi^{\epsilon}\xi^{-\epsilon}\sigma_0^{\epsilon}\xi^{\epsilon}\xi^\epsilon\\
\overset{(9)}{=}~&\sigma_0^{\epsilon}\sigma_0^{\epsilon}\\
\overset{(8)}{=}~&t_{12}(1).
\end{align*}
Hence $t_{12}(1)\in CC$ and therefore $T\subseteq CC$. Thus we have shown $(i)\Leftrightarrow (ii)$.\\
Applying the equivalence $(i)\Leftrightarrow (ii)$ to $C^{-1}$ we obtain 
\[\det(C^{-1})^2=1\Leftrightarrow T\subseteq C^{-1}C^{-1}.\]
But clearly
\[\det(C^{-1})^2=1\Leftrightarrow \det(C)^2=1.\]
Thus we have shown $(i)\Leftrightarrow (iii)$.
\end{proof}

\begin{lemma}\label{lemperm}
Let $a,d,f,x\in K^*$ and $b,c\in K$. Then
\[\begin{pmatrix}&&a\\d&&b\\&f&c\end{pmatrix}\sim_E\begin{pmatrix}&&d\\f&&a^{-1}bf\\&a&c\end{pmatrix}\sim_E\begin{pmatrix}&&ax^2\\dx^{-1}&&bx\\&fx^{-1}&c\end{pmatrix}.\]
\end{lemma}
\begin{proof}
A straightforward computation shows that
\begin{equation*}
\begin{pmatrix}&&a\\d&&b\\&f&c\end{pmatrix}^{\epsilon}=\begin{pmatrix}&&d\\f&&a^{-1}bf\\&a&c\end{pmatrix}
\end{equation*}
where $\epsilon=\hat p_{32}\hat p_{31}t_{23}(-a^{-1}c)t_{13}(a^{-1}b)d_{13}(-1)\in E$. Moreover,
\begin{equation*}
\begin{pmatrix}&&a\\d&&b\\&f&c\end{pmatrix}^{d_{31}(x)}=\begin{pmatrix}&&ax^2\\dx^{-1}&&bx\\&fx^{-1}&c\end{pmatrix}.
\end{equation*}
\end{proof}

\begin{proposition}\label{propcuben=3}
Let $C$ be an $H$-class such that $\chi_C$ does not have a root in $K$. Then the following are equivalent.
\begin{enumerate}[(i)]
\item $\det(C)^3=1$. 
\item $T\subseteq CCC$.
\item $T\subseteq C^{-1}C^{-1}C^{-1}$.
\end{enumerate}
\end{proposition}
\begin{proof}
We will first show  $(i)\Leftrightarrow (ii)$ and then $(i)\Leftrightarrow (iii)$.\\
The implication $(ii)\Rightarrow (i)$ is obvious since $\det(T)=1$. Suppose now that $\det(C)^3=1$. Choose a $\sigma\in C$. By Lemma \ref{lem3} we may assume that $\sigma$ is in Frobenius form. Since $\chi_C$ does not have a root in $K$, it is irreducible (because $n=3$). Hence the matrix $\sigma$ has only one invariant factor and therefore 
\[\sigma=\begin{pmatrix}&&a\\1&&b\\&1&c\end{pmatrix}\]
for some $a,b,c\in K$.
Note that $a=\det(\sigma)$ and hence $a^3=1$.

\begin{enumerate}[{\bf Case} 1] 

\item Suppose that $b\neq 0$. Set
\[\sigma_0=\begin{pmatrix}&&a\\1&&\\&1&\end{pmatrix}.\]
Then
\[(\sigma_0^2)^{\hat p_{32}}=\begin{pmatrix}&&-a\\1&&\\&-a&\end{pmatrix}.\]
It follows from Lemma \ref{lemperm} that there is an $\epsilon\in E$ such that
\begin{equation}
(\sigma_0^2)^{\epsilon}=\begin{pmatrix}&&a^2\\1&&\\&1&\end{pmatrix}.
\end{equation}
Set 
\[\xi=\begin{pmatrix}-1&&-b\\&-1&c\\&&1\end{pmatrix}.\]
Clearly $\xi=t_{13}(-b)t_{23}(c)d_{12}(-1)=d_{12}(-1)t_{13}(b)t_{23}(-c)\in E$ and
\begin{equation}
\xi^2=e.
\end{equation}
One checks easily that 
\begin{equation}
\sigma_0\xi=\sigma^{d_{13}(-1)}.
\end{equation}
Clearly 
\begin{align}
&\sigma^{d_{13}(-1)\epsilon}\sigma^{d_{13}(-1)\epsilon\xi^\epsilon}\notag\\
\overset{(13)}{=}~&(\sigma_0\xi)^{\epsilon}(\sigma_0\xi)^{\epsilon\xi^\epsilon}\notag\\
=~&\sigma_0^{\epsilon}\xi^{\epsilon}\xi^{-\epsilon}\sigma_0^{\epsilon}\xi^{\epsilon}\xi^\epsilon\notag\\
\overset{(12)}{=}~&\sigma_0^{\epsilon}\sigma_0^{\epsilon}\notag\\
\overset{(11)}{=}~&\begin{pmatrix}&&a^2\\1&&\\&1&\end{pmatrix}.
\end{align}
Our next to goal is to show that there are $\epsilon',\epsilon''\in E$ such that 
\[(t_{12}(1)(\sigma^{-1})^{\epsilon'})^{\epsilon''}=\begin{pmatrix}&&a^2\\1&&\\&1&\end{pmatrix},\]
which will finish Case 1 in view of (14).
Clearly 
\[\sigma^{-1}=\begin{pmatrix}-ba^{-1}&1&\\-ca^{-1}&&1\\a^{-1}&&\end{pmatrix}.\]
A straightforward computation shows that 
\[(\sigma^{-1})^{\epsilon'}=\begin{pmatrix}-ba^{-1}+b^{-1}c&&-b\\ba^{-1}&&\\b^{-3}c^2&-b^{-2}&-b^{-1}c\end{pmatrix}\]
where $\epsilon'=\hat p_{23}t_{31}(b^{-1}c)d_{32}(-b)\in E$. It follows that
\[(t_{12}(1)(\sigma^{-1})^{\epsilon'})^{t_{31}(b^{-2}c)}=\begin{pmatrix}&&-b\\ba^{-1}&&\\&-b^{-2}&\end{pmatrix}\]
whence, by Lemma \ref{lemperm}, there is an $\epsilon''\in E$ such that 
\begin{equation}
(t_{12}(1)(\sigma^{-1})^{\epsilon'})^{\epsilon''}=\begin{pmatrix}&&a^2\\1&&\\&1&\end{pmatrix}.
\end{equation}
By (14) and (15) we have 
\[\sigma^{d_{13}(-1)\epsilon(\epsilon'')^{-1}}\sigma^{d_{13}(-1)\epsilon\xi^\epsilon(\epsilon'')^{-1}}\sigma^{\epsilon'}=t_{12}(1).\]
Thus $T\subseteq CCC$.

\item Suppose that $b=0$ and $c\neq 0$. Let $\sigma_0$, $\epsilon$ and $\xi$ be as in Case 1. We will show that there are $\epsilon',\epsilon''\in E$ such that 
\[(t_{12}(1)(\sigma^{-1})^{\epsilon'})^{\epsilon''}=\begin{pmatrix}&&a^2\\1&&\\&1&\end{pmatrix},\]
which will finish Case 2 in view of (14).
Clearly 
\[\sigma^{-1}=\begin{pmatrix}&1&\\-ca^{-1}&&1\\a^{-1}&&\end{pmatrix}\]
and hence 
\[(\sigma^{-1})^{\epsilon'}=\begin{pmatrix}&c^{-1}&\\&&c^2\\c^{-1}a^{-1}&-c^{-1}a^{-1}&\end{pmatrix}\]
where $\epsilon'=t_{31}(ca^{-1})d_{32}(c)\in E$. It follows that
\[(t_{12}(1)(\sigma^{-1})^{\epsilon'})^{t_{12}(1)\hat p_{23}}=\begin{pmatrix}&&c^{-1}\\-c^{-1}a^{-1}&&\\&-c^{2}&\end{pmatrix}\]
whence, by Lemma \ref{lemperm}, there is an $\epsilon''\in E$ such that 
\begin{equation}
(t_{12}(1)(\sigma^{-1})^{\epsilon'})^{\epsilon''}=\begin{pmatrix}&&a^2\\1&&\\&1&\end{pmatrix}.
\end{equation}
By (14) and (16) we have 
\[\sigma^{d_{13}(-1)\epsilon(\epsilon'')^{-1}}\sigma^{d_{13}(-1)\epsilon\xi^\epsilon(\epsilon'')^{-1}}\sigma^{\epsilon'}=t_{12}(1).\]
Thus $T\subseteq CCC$.

\item Suppose that $b=c=0$. Set
\[\sigma_0=\begin{pmatrix}&&a\\1&&1\\&1&\end{pmatrix}.\]
Then
\[(\sigma_0^2)^{\hat p_{32}t_{23}(-a^{-1})}=\begin{pmatrix}&&-a\\1&&a^{-1}\\&-a&2\end{pmatrix}.\]
It follows from Lemma \ref{lemperm} that there is an $\epsilon\in E$ such that
\begin{equation}
(\sigma_0^2)^{\epsilon}=\begin{pmatrix}&&a^2\\1&&-1\\&1&2\end{pmatrix}.
\end{equation}
Set 
\[\xi=\begin{pmatrix}-1&&-1\\&-1&\\&&1\end{pmatrix}.\]
Clearly $\xi=t_{13}(-1)d_{12}(-1)=d_{12}(-1)t_{13}(1)\in E$ and
\begin{equation}
\xi^2=e.
\end{equation}
One checks easily that 
\begin{equation}
\sigma_0\xi=\sigma^{d_{13}(-1)}.
\end{equation}
Clearly 
\begin{align}
&\sigma^{d_{13}(-1)\epsilon}\sigma^{d_{13}(-1)\epsilon\xi^\epsilon}\notag\\
\overset{(19)}{=}~&(\sigma_0\xi)^{\epsilon}(\sigma_0\xi)^{\epsilon\xi^\epsilon}\notag\\
=~&\sigma_0^{\epsilon}\xi^{\epsilon}\xi^{-\epsilon}\sigma_0^{\epsilon}\xi^{\epsilon}\xi^\epsilon\notag\\
\overset{(18)}{=}~&\sigma_0^{\epsilon}\sigma_0^{\epsilon}\notag\\
\overset{(17)}{=}~&\begin{pmatrix}&&a^2\\1&&-1\\&1&2\end{pmatrix}.
\end{align}
Our next to goal is to show that there are $\epsilon',\epsilon''\in E$ such that 
\[(t_{12}(1)(\sigma^{-1})^{\epsilon'})^{\epsilon''}=\begin{pmatrix}&&a^2\\1&&-1\\&1&2\end{pmatrix},\]
which will finish Case 3 in view of (20).
Clearly 
\[\sigma^{-1}=\begin{pmatrix}&1&\\&&1\\a^{-1}&&\end{pmatrix}\]
and hence
\[(\sigma^{-1})^{\epsilon'}=\begin{pmatrix}&-a^{-1}&2a^{-2}\\2&&-a^{-1}\\a&&\end{pmatrix}\]
where $\epsilon'=d_{13}(-a)t_{23}(-2a^{-1})\in E$. It follows that
\[(t_{12}(1)(\sigma^{-1})^{\epsilon'})^{t_{23}(2a^{-1})t_{13}(2a^{-1})\hat p_{21}}=\begin{pmatrix}&&-a^{-1}\\a^{-1}&&a^{-1}\\&-a&2\end{pmatrix}\]
whence, by Lemma \ref{lemperm}, there is an $\epsilon''\in E$ such that 
\begin{equation}
(t_{12}(1)(\sigma^{-1})^{\epsilon'})^{\epsilon''}=\begin{pmatrix}&&a^2\\1&&-1\\&1&2\end{pmatrix}.
\end{equation}
By (20) and (21) we have 
\[\sigma^{d_{13}(-1)\epsilon(\epsilon'')^{-1}}\sigma^{d_{13}(-1)\epsilon\xi^\epsilon(\epsilon'')^{-1}}\sigma^{\epsilon'}=t_{12}(1).\]
Thus $T\subseteq CCC$.

\end{enumerate}
We have shown $(i)\Leftrightarrow (ii)$. Applying the equivalence $(i)\Leftrightarrow (ii)$ to $C^{-1}$ we obtain 
\[\det(C^{-1})^3=1\Leftrightarrow T\subseteq C^{-1}C^{-1}C^{-1}.\]
But clearly
\[\det(C^{-1})^3=1\Leftrightarrow \det(C)^3=1.\]
Thus we have shown $(i)\Leftrightarrow (iii)$.
\end{proof}

The theorem below follows from Lemma \ref{lemm=1} and Propositions \ref{propfour}, \ref{proprootn=3}, \ref{propsquaren=3}, \ref{propcuben=3}.

\begin{theorem}\label{thmn=3}
Let $C$ be a noncentral $H$-class. Then the following holds.
\begin{enumerate}[(i)]
\item If $C=T$, then $m(C)=1$.
\item If $C\neq T$ and $\chi_C$ has a root in $K$, then $m(C)=2$. In this case $T\subseteq CC^{-1}$.
\item If $C\neq T$, $\chi_C$ has no root in $K$ and $\det(C)^2=1$, then $m(C)=2$. In this case $T\subseteq CC$.
\item If $C\neq T$, $\chi_C$ has no root in $K$, $\det(C)^2\neq 1$ and $\det(C)^3=1$, then $m(C)=3$. In this case $T\subseteq CCC$.
\item If $C\neq T$, $\chi_C$ has no root in $K$, $\det(C)^2\neq 1$ and $\det(C)^3\neq 1$, then $m(C)=4$. In this case $T\subseteq CC^{-1}CC^{-1}$.
\end{enumerate}
\end{theorem}
\newpage

\begin{example}
Suppose that $K=\mathbb{F}_2$ or $K=\mathbb{F}_3$. Then $m(C)\leq 2$ for any noncentral $H$-class $C$ since $a^2=1$ for any $a\in K^*$.
\end{example}

\begin{example}
Suppose that $K=\mathbb{F}_5$. Let $C$ be the $H$-class of $[X^3-X+2]$. Then $C\neq T$ by Lemma \ref{lemfroel}, $\chi_C=X^3-X+2$ has no root in $K$, $\det(C)^2=-1\neq 1$ and $\det(C)^3=2\neq 1$. Hence $m(C)=4$.
\end{example}

\begin{example}
Suppose that $K=\mathbb{Q}$. Let $C$ be the $H$-class of $[X^3-2]$. Then $C\neq T$ by Lemma \ref{lemfroel}, $\chi_C=X^3-2$ has no root in $K$, $\det(C)^2=4\neq 1$ and $\det(C)^3=8\neq 1$. Hence $m(C)=4$.
\end{example}

\begin{example}
Suppose that $K=\mathbb{R}$. Then $m(C)\leq 2$ for any noncentral $H$-class $C$ since any cubic polynomial with real coefficients has at least one real root.
\end{example}

\section{Products of conjugacy classes in $\GL_\infty(K)$}
Set $G:=\GL_\infty(K)$ and $E:=\E_\infty(K)$. $H$ denotes a subgroup of $G$ containing $E$, and $T$ denotes the $H$-class of $[t_{2,1,2}(1)]_\infty$. We will determine $m(C)$ for any noncentral $H$-class $C$ ($m(C)$ is defined as in Section 1). Note that the center of $G$ consists only of $[e]_\infty$. Hence there is only one central $H$-class.

Our first goal is to define the Frobenius form of an element of $G$. 

\begin{lemma}\label{lemfrostable}
Let $n\in \N$ and $\sigma\in \GL_n(K)$. If $F(\sigma)=[P_1]\oplus\dots\oplus[P_r]$, then \[F(1\oplus \sigma)=[P_1]\oplus\dots\oplus[P_{i-1}]\oplus [P_{i}(X-1)]\oplus[P_{i+1}]\oplus\dots\oplus[P_r]\]
where $i=\min\{j\in\{0,\dots,r\}\mid P_j(X-1)\textnormal{ divides }P_{j+1}\}$. Here we set $P_0:=1$ and $P_{r+1}:=P_r(X-1)$.
\end{lemma}
\begin{proof}
Clearly the characteristic matrix of $1\oplus \sigma$ is the matrix $A=(X-1)\oplus (Xe_{n\times n}-\sigma)$. Since the invariant factors of $\sigma$ are $P_1,\dots,P_r$, the matrix $A$ is equivalent to the matrix $B=(X-1)\oplus 1\oplus\dots\oplus 1\oplus P_1\oplus\dots\oplus P_r$. By applying the algorithm described in \cite[Part V, Chapter 20, Proof of Theorem 3.2]{bjn} we get that $B$ is equivalent to the matrix $C=1\oplus\dots\oplus 1\oplus P_1\oplus\dots\oplus P_{i-1}\oplus P_i(X-1)\oplus P_{i+1}\dots\oplus P_r$. Clearly $C$ is the Smith normal form of $A$ and hence $P_1,\dots,P_{i-1},P_i(X-1),P_{i+1},\dots P_r$ are the invariant factors of $1\oplus \sigma$.
\end{proof}

For any $\sigma\in G$ there is a minimal $n_\sigma\in\N$ such that $\sigma$ has a representative in $\GL_{n_\sigma}(K)$. For any $n\geq n_\sigma$ we write $\sigma^{(n)}$ for the unique representative of $\sigma$ in $\GL_n(K)$. Proposition \ref{propfrostable} below shows that the sequence 
\[F(\sigma^{(n_\sigma)}),F(\sigma^{(n_\sigma+1)}),F(\sigma^{(n_\sigma+2)})\dots\]
eventually stabilises (up to the equivalence relation $\sim_\infty$).

\begin{proposition}\label{propfrostable}
Let $\sigma\in G$. Then there is a $s\geq n_\sigma$ such that 
$[F(\sigma^{(s)})]_\infty=[F(\sigma^{(t)}]_\infty$ for any $t\geq s$.
\end{proposition}
\begin{proof}
If $F(\sigma^{(n_\sigma)})=[P_1]\oplus\dots\oplus[P_r]$, then by Lemma \ref{lemfrostable}, \[F(\sigma^{(n_\sigma+1)})=[P_1]\oplus\dots\oplus[P_{i-1}]\oplus [P_{i}(X-1)]\oplus[P_{i+1}]\oplus\dots\oplus[P_r]\]
where $i=\min\{j\in\{0,\dots,r\}\mid P_j(X-1)\textnormal{ divides }P_{j+1}\}$. Since $P_{i-1}(X-1)$ divides $P_i(X-1)$ but $P_j(X-1)$ does not divide $P_{j+1}$ for any $j\in\{0,\dots,i-2\}$, we get 
\[F(\sigma^{(n_\sigma+2)})=[P_1]\oplus\dots\oplus[P_{i-2}]\oplus [P_{i-1}(X-1)]\oplus [P_{i}(X-1)]\oplus[P_{i+1}]\oplus\dots\oplus[P_r],\]
again by Lemma \ref{lemfrostable}. After repeating this step a finite number of times we arrive at
\[F(\sigma^{(n_\sigma+i)})=[P_1(X-1)]\oplus\dots\oplus [P_{i}(X-1)]\oplus[P_{i+1}]\oplus\dots\oplus[P_r].\]
Hence, again by Lemma \ref{lemfrostable},
\[F(\sigma^{(n_\sigma+t)})=[X-1]\oplus\dots\oplus [X-1]\oplus[P_1(X-1)]\oplus\dots\oplus [P_{i}(X-1)]\oplus[P_{i+1}]\oplus\dots\oplus[P_r].\]
for any $t\geq i$. Since $[X-1]$ is the $1\times 1$ matrix $(1)$, the assertion of the proposition follows.
\end{proof}

Let $\sigma\in G$. By Proposition \ref{propfrostable} there is a $s\geq n_\sigma$ such that $[F(\sigma^{(s)})]_\infty=[F(\sigma^{(t)}]_\infty$ for any $t\geq s$. We define the \textit{Frobenius form} $F(\sigma)$ of $\sigma$ by $F(\sigma)=[F(\sigma^{(s)})]_\infty$. Clearly $F(\sigma)$ is well-defined. 

\begin{proposition}\label{propsimstable}
Let $\sigma,\tau\in G$. Then $\sigma\sim \tau$ iff $F(\sigma)=F(\tau)$.
\end{proposition}
\begin{proof}
Choose an $s\geq n_\sigma,n_\tau$ such that $[F(\sigma^{(s)})]_\infty=[F(\sigma^{(t)})]_\infty$ and $[F(\tau^{(s)})]_\infty=[F(\tau^{(t)})]_\infty$ for any $t\geq s$. Then $F(\sigma)=[F(\sigma^{(s)})]_\infty$ and $F(\tau)=[F(\tau^{(s)})]_\infty$. One checks easily that
\begin{align*}
&\sigma\sim \tau\\
\Leftrightarrow&\exists t\geq s:\sigma^{(t)}\sim\tau^{(t)}\text{ in }\GL_t(K)\\
\Leftrightarrow&\exists t\geq s:F(\sigma^{(t)})=F(\tau^{(t)})\\
\Leftrightarrow&F(\sigma^{(s)})=F(\tau^{(s)})\\
\Leftrightarrow&F(\sigma)=F(\tau).
\end{align*}
\end{proof}

By Proposition \ref{propsimstable} we can define the Frobenius form $F(C)$ of an $H$-class $C$ in the obvious way. Below we compute $F(T)$. Note that $T=\{g\in G\mid F(\sigma)=F(T)\}$ by Lemma \ref{lem2}.

\begin{lemma}
$F(T)=[[(X-1)^2]]_\infty$.
\end{lemma}
\begin{proof}
It is an easy exercise to show that 
\[F(t_{2,1,2}(1))=\left(\begin{array}{cc}&-1\\1&2\end{array}\right)=[(X-1)^2].\]
Hence $F(e_{n\times n}\oplus t_{2,1,2}(1))=[X-1]\oplus\dots\oplus[X-1]\oplus [(X-1)^2]$ for any $n\in\N$, by Lemma \ref{lemfrostable}. The assertion of the lemma follows.
\end{proof}

\begin{lemma}\label{lemm=1stab}
Let $C$ be a noncentral $H$-class. Then $m(C)=1$ iff $C=T$.
\end{lemma}
\begin{proof}
Clear since $T=T^{-1}$ by Lemma \ref{lem1}.
\end{proof}

\begin{lemma}\label{lemrootstab}
Let $C$ be a noncentral $H$-class. Then $T\subseteq CC^{-1}$.
\end{lemma}
\begin{proof}
Choose a $\sigma\in C$ and an $n>n_\sigma$. Then clearly $\sigma^{(n)}$ is noncentral and  $\chi_{\sigma^{(n)}}$ has a linear factor. By Proposition \ref{proproot} we get $t_{n,n-1,n}(1)\in C^{(n)}(C^{(n)})^{-1}$ where $C^{(n)}$ is the $E_n(K)$-class of $\sigma^{(n)}$. It follows that $[t_{2,1,2}(1)]_\infty\in CC^{-1}$ and thus $T\subseteq CC^{-1}$.
\end{proof}

Theorem \ref{thmstab} below follows directly from Lemmas \ref{lemm=1stab} and \ref{lemrootstab}.
\begin{theorem}\label{thmstab}
Let $C$ be a noncentral $H$-class. Then $m(C)=1$ if $C=T$ respectively $m(C)=2$ if $C\neq T$.
\end{theorem}


\begin{thebibliography}{99}














\bibitem{bass} H. Bass, \emph{K-theory and stable algebra}, Publ. Math. Inst. Hautes \'Etudes Sci. {\bf 22} (1964), 5--60.

\bibitem{bjn}  P.B. Bhattacharya, S.K. Jain, S.R. Nagpaul, Basic abstract algebra, 2nd ed., Cambridge University Press, 1994.

\bibitem{brenner} J.L. Brenner, \emph{The linear homogeneous group, III}, Ann. of Math. {\bf 71} (1960), no. 2, 210--223.















\bibitem{golubchik} I.Z. Golubchik, \emph{On the general linear group over an associative ring}, Uspekhi Mat. Nauk {\bf 28} (1973), no. 3, 179--180 (Russian).



















\bibitem{preusser2} R. Preusser, \emph{Sandwich classification for $GL_n(R)$, $O_{2n}(R)$ and $U_{2n}(R,\Lambda)$ revisited}, J. Group Theory. {\bf 21} (2018), no. 1, 21--44.









\bibitem{preusser11} R. Preusser, \emph{Reverse decomposition of unipotents over noncommutative rings I: General linear groups}, Linear Algebra Appl. {\bf 601} (2020), 285--300.





\bibitem{vaserstein} L.N. Vaserstein, \emph{On the normal subgroups of $GL_n$ over a ring}, Lecture Notes in Math. {\bf 854} (1981), 454--465.

\bibitem{vaserstein_ban} L.N. Vaserstein, \emph{Normal subgroups of the general linear groups over Banach algebras}, J. Pure Appl. Algebra {\bf 41} (1986), 99--112.

\bibitem{vaserstein_neum} L.N. Vaserstein, \emph{Normal subgroups of the general linear groups over von Neumann regular rings}, Proc. Am. Math. Soc {\bf 96} (1986), no. 2, 209--214.


\bibitem{wilson} J.S. Wilson, \emph{The normal and subnormal structure of general linear groups}, Math. Proc. Cambridge Philos. Soc. {\bf 71} (1972), 163--177.

\end{thebibliography}
\end{document}